\documentclass{article}

\usepackage{amsfonts}
\usepackage{latexsym}
\usepackage{amssymb}
\usepackage{amsmath}
\usepackage{times}
\usepackage{amsthm}
\usepackage{mathrsfs}
\usepackage{graphics}
\usepackage{hyperref}

 \newcommand{\eps}{\varepsilon}
 \newcommand{\ovl}{\overline}
\newcommand{\C}{\mathbb{C}}
 \newcommand{\G}{\mathcal{G}}
 \newcommand{\F}{\mathcal{F}}
 \newcommand{\Og}{\mathcal{O}}
 \newcommand{\Z}{\mathbb{Z}}
 \newcommand{\U}{\mathscr{U}}
 \newcommand{\Pc}{\mathcal{P}}
 \newcommand{\Q}{\mathbb{Q}}
 \newcommand{\ds}{\displaystyle}
 \renewcommand{\P}{\mathbb{P}}
\newcommand{\Ce}{\overline{\mathbb{C}}}
\DeclareMathOperator{\End}{End}
\DeclareMathOperator{\sing}{Sing} 
\DeclareMathOperator{\gen}{gen}
\DeclareMathOperator{\tang}{Tang}
\DeclareMathOperator{\aut}{Aut}
\DeclareMathOperator{\Kod}{Kod}
\DeclareMathOperator{\Res}{Res}
\DeclareMathOperator{\Tr}{Tr}
\DeclareMathOperator{\Hol}{Hol}
\DeclareMathOperator{\diff}{Diff}

\newtheorem*{theoremx}{Theorem}

\newtheorem{theorem}{Theorem}[section]
\newtheorem{proposition}[theorem]{Proposition}
\newtheorem{corollary}[theorem]{Corollary}
\newtheorem{lemma}[theorem]{Lemma}

\theoremstyle{definition}
\newtheorem{definition}[theorem]{Definition}
\newtheorem{example}[theorem]{Example}

\theoremstyle{remark}
\newtheorem{remark}[theorem]{Remark}

\title{Classification of flat pencils of foliations on compact complex surfaces}

\author{Liliana Puchuri\thanks{Pontificia Universidad Cat\'olica del Per\'u \& Instituto de Mat\'ematica and Ciencias Afines (IMCA). Email: \texttt{lpuchuri@pucp.pe}}}

\begin{document}

\maketitle

\begin{abstract}
Related to the classification of regular foliations in a complex algebraic surface, we address the problem of classifying the complex surfaces which admit a flat pencil of foliations.  On this matter, a classification of flat pencils which admit foliations with a first integral of genus one and isolated singularities was done by Lins Neto. In this work, we complement Lins Neto's work, by obtaining the classification of compact complex surfaces which have a pencil with an invariant tangency set. 



\end{abstract}

\section{Introduction}
It is a well known fact that every foliation on $\P^2$ must have singularities, however, there exist complex surfaces which admit regular foliations, that is, foliations without singularities. Then it is natural to ask for a classification of complex surfaces which admit regular foliations.
Related to this problem, we have a classification of regular foliations on complex surfaces due to Brunella~\cite{BR3}, which is based from Enriques-Kodaira classification of compact complex surfaces. Brunella's classification is as follows.
\begin{theorem}[{\cite[Th\'eor\`eme~2]{BR3}}]\label{teo:brintro}
Every regular foliation over a compact complex surface $X$ with $\Kod(X)<2$ belongs to the following list:
\begin{enumerate}
 \item elliptic or rational fibrations,
 \item foliations transversal to elliptic or rational fibrations,
 \item turbulent foliations,
 \item linear foliations on a torus,
 \item trivial foliations on Hopf or Inoue surfaces.
\end{enumerate}
\end{theorem}

The notion of \emph{pencil} (or \emph{linear family}) of foliations was defined by Lins Neto~\cite{LN2}. A pencil $\Pc:=\{\F_{\alpha}\}_{\alpha \in \Ce}$ on a compact complex surface $X$ is defined by two foliations $\F_0$ and $\F_{\infty}$ in $\Pc$ together with an isomorphism between their normal bundles.
Associated to a pencil are its \emph{curvature} (see Section~\ref{sec:curvature}) and its \emph{tangency set} $\Delta(\Pc)$ (see Section~\ref{sec:definitions}) which is an analytic set formed by the singularities of the foliations in $\Pc$. A pencil is called \emph{flat} if its curvature is zero.

In this work, we address the following problem:
\begin{quote}
``Classify the compact complex surfaces which admit flat pencils with an invariant tangency set''.
\end{quote}
Moreover, for these particular pencils, we will characterize the following set
\[
I_p(\Pc):=\{\alpha \in \Ce\::\:\F_{\alpha}\mbox{ has a meromorphic first integral on } X\}.
\]
This is related to the Poincar\'e's problem, which consists in bounding the degree of an invariant algebraic curve in terms of the degree of the foliation.

The paper is divided as follows: Sections~\ref{sec:prelim} and \ref{sec:pencils} contain the preliminary definitions we will need in this work, in particular, Section~\ref{sec:pencils} contains the definition and main features of the pencils of foliations.
In Section~\ref{sec:caseempty}, we consider the case when $\Delta(\Pc)$ is empty. Theorems~\ref{teodeltavacio} and \ref{teo:linearfol}, and Corollaries~\ref{coro:ipptoro} and~\ref{coro:ipphopf}, are comprised in the following theorem.
\begin{theoremx}
Let $X$ be a complex compact surface which admits a pencil $\Pc$ with empty tangency set. Then 
 $X$ is either a torus or a Hopf surface and $\Pc$ is generated by linear foliations.
Moreover
\begin{enumerate}
\item If $X$ is a Hopf surface then $I_p(\Pc)=\emptyset$.
\item If $X$ is a torus and  $\# I_p(\Pc)\geq 3$ then $X=E\times E$, with $E=\C/\langle 1,\tau\rangle$ and $I_p(\Pc)\setminus\{\infty\}$ is either $\Q$ or $\Q(\tau)$, only up to a reparametrization of the parameter space of the pencil.
\end{enumerate}
\end{theoremx}
Finally, in Section~\ref{sec:invariant}, we deal with flat pencils with invariant and non-empty tangency set.  
Four remarkable examples in $\P^2$ of this kind of pencils were given by Lins Neto in~\cite{LN3,LN1,LN2}, see Example~\ref{ejmmodelo}. Lins Neto also obtained in~\cite{LN3} the following characterization of these pencils.
\begin{theorem}[{\cite[Theorem~3]{LN3}}]\label{teo:linsneto}
Let $M$ be a compact complex surface and $\F$, $\G$ be two foliations on $M$ such that $T_{\F}=T_{\G}$ and $\Pc=(\F_{\alpha})_{\alpha\in\Ce}$ be the pencil generated by $\F$ and $\G$. Suppose that
\begin{enumerate}
 \item $\F\neq \G$;
 \item The singularities of $\F$ are reduced in the sense of Seidenberg;
 \item $\F$ and $\G$ have holomorphic first integrals, say $f:M\to S_1$ and $g:M\to S_2$, respectively, where $f$ is an elliptic fibration.
\end{enumerate}
Then
\begin{enumerate}
 \item If $K_M\neq 0$ then $M$ is a rational surface. In this case, the pencil is  bimeromorphically equivalent to one of the types of Example~\ref{ejmmodelo}. Moreover, we have $I_p(\Pc)=\lambda\cdot\Q\cdot\Gamma\cup\{\infty\}$, where $\lambda\in\C^*$ and $\Gamma=\langle 1,e^{2\pi i/3}\rangle$ or $\langle 1,i \rangle$. In particular, $E(\Pc)$ is countable and dense in $\Ce$.
 \item If $K_M=0$ then, either $M$ is a complex algebraic torus, or $M$ is an algebraic $K3$ surface. Moreover, the family is exceptional (in the sense of~\cite{LN3}) if, and only if, $I_p(\Pc)$ contains at least three elements.
\end{enumerate}
\end{theorem}

In~\cite{LN1,LN2}, Lins Neto computes the set $I_p(\Pc)$, for the pencils given in Example~\ref{ejmmodelo}. This was done using the theory of foliations transversal to the fibers of a fibration applied to a certain element of the pencil with an holomorphic first integral with genus one, in order to find an explicit formula of the generators of the holonomy group of the foliations in such pencils.  In this work, we deal with the case when the pencil $\Pc$ has an element with an holomorphic first integral with genus zero. Our first result is about the generators of the holonomy group of the foliations $\F_{\alpha}\in\Pc$ with isolated singularities. In the following theorem, $IS(\Pc)$ is the set of indices where the associated foliation has isolated singularities.

\begin{theoremx}
Let $\Pc=\{\F_{\alpha}\}_{\alpha\in\Ce}$ be a flat pencil on a compact complex surface $X$ such that $\F_{\infty}$ has an holomorphic first integral $f:X\to \P^1$ and $\Delta(\Pc)$ is invariant. If $\gen(f)=0$ then, for any $\alpha\in IS(\Pc)$, the global holonomy group $G_{\alpha}$ of $\F_{\alpha}$ is finitely generated by $f_{1,\alpha},\ldots,f_{k,\alpha}$, where these generators could be either
\[
f_{j,\alpha}(z)=\lambda_jz+a_j\alpha +b_j
,\qquad j=1,\ldots,k,
\]
or
\[
f_{j,\alpha}(z)=\exp(2\pi i(\mu_j\alpha+\nu_j))z,\qquad j=1,\ldots,k.
\]
\end{theoremx}
The following theorem is our main result. In this theorem, under the same conditions of the previous theorem, we prove that $X$ is a rational surface  and characterize the set $I_p(\Pc)$. 

\begin{theoremx}
Let $\Pc=\{\F_{\alpha}\}_{\alpha\in\Ce}$ be a flat pencil on a compact complex surface $X$ such that $\F_{\infty}$ has an holomorphic first integral $f:X\to \P^1$ and $\Delta(\Pc)$ is invariant. If $\gen(f)=0$ then $X$ is a rational surface where one of the following hold:
\begin{enumerate}
 \item $I_p(\Pc)$ is finite,
 \item $IS(\Pc)\subset I_p(\Pc)$, or
 \item $I_p(\Pc)\cap IS(\Pc)=\Q\cap IS(\Pc)$, up to a reparametrization of the pencil.
\end{enumerate}

\end{theoremx}

\section{Preliminaries}\label{sec:prelim}
In this section we recall some basic concepts needed for later sections.  
Throughout this section, $X$ will denote a compact complex surface, unless otherwise stated.

An \emph{holomorphic foliation} $\F$ on $X$ is given by a family of holomorphic 1-forms $\{\omega_i\}_{i\in I}$ defined over an open covering $\U=\{U_i\}_{i\in I}$ of $X$ such that $\omega_i= g_{ij}\omega_j$,
with $g_{ij}\in \Og^*(U_{ij})$, where $U_{ij}=U_i \cap U_j\neq \emptyset$.
The \emph{singular set} of $\F$, is the analytic set $\sing(\F)$ such that $\sing(\F)\cap U_i=\{w_{i}=0\}$.	
The \emph{multiplicative cocycles} $\{g_{ij}\}_{i,j\in I}$ define the linear fibration $N_{\F}$ on $X$, which will be called the \emph{normal fibration} of $\F$.

A foliation $\F$ can also be defined by a family of holomorphic vector fields $X_i$, defined over an open covering $\U=\{U_i\}_{i\in I}$ of $X$ such that $X_i= f_{ij}X_j$, with $f_{ij}\in \Og^*(U_{ij})$, where $U_{ij}=U_i \cap U_j\neq \emptyset$.
Thus the \emph{tangent bundle} of $\F$ on $X$, denoted by $T_{\F}$, is defined by the multiplicative cocycles $\{f_{ij}^{-1}\}_{i,j\in I}$.

Let $N_{\F}^*$ and $T_{\F}^*$ be the dual fibrations of $N_{\F}$ and $T_{\F}$, respectively. Then~\cite[Lemme~1]{BR3}:
\[
K_X=N_{\F}^*\otimes T_{\F}^*,
\]
where $K_X$ is the \emph{canonical bundle} of $X$. 

We now recall the \emph{Poincar\'e-Hopf} and \emph{Baum-Bott} indices, defined by Brunella~\cite{BR2}, see also~\cite{MR0261635,BR3,BR1}.
Let $p\in X$ be an isolated singularity of $\F$ and let $(x,y,U)$ be a coordinate system with $p\in U$ and $x(p)=y(p)=0$, such that $\F$ is represented in $U$ by
\[
Y(x,y)=P(x,y)\dfrac{\partial}{\partial x}+Q(x,y)\dfrac{\partial}{\partial y},
\]
where $P$ and $Q$ are holomorphic over $U$ and $\gcd(P,Q)=1$.  In addition, let $J$ be the Jacobian matrix of $(P,Q)$ at $(0,0)$.  The \emph{Poincar\'e-Hopf index} of $\F$ at $p$ is defined as
\[
PH(\F,p)=\Res_{(0,0)}\dfrac{\det J}{P\cdot Q}dx\wedge dy.
\]
It is well known that this index coincides with the \emph{Milnor number} of $Y$ at $p$, that is,
\[
PH(\F,p)=\dim_{\C}\dfrac{\Og_p}{\langle P,Q\rangle}.
\]
The \emph{Baum-Bott index} of $\F$ at $p$, is defined as
\[
BB(\F,p)=\Res_{(0,0)}\dfrac{(\Tr J)^2}{P\cdot Q}dx\wedge dy.
\]
Denote $m(\F)=\ds\sum_{p\in\sing(\F)}PH(\F,p)$ and $BB(\F)=\ds\sum_{p\in\sing(\F)}BB(\F,p)$. 

The following formulas can be found in \cite[\S 2]{BR3}, Proposition~1, and \cite[p.~34]{BR1}, respectively:
\begin{gather}
m(\F)=c_2(X)+BB(\F)-N_{\F}\cdot K_X,\label{eq:newmf}\\
BB(\F)=N_{\F}\cdot N_{\F}.\label{eq:newbbf}
\end{gather}

Let $C\subset X$ a curve which is non-invariant by $\F$. Given $p\in C$, let $f=0$ a local reduced equation of $C$ and  $\F$ represented by  $Y$ an holomorphic field  in some neighborhood $U$ of $p$.
The \emph{tangency index} of $\F$ with respect to $C$ at $p$ is defined as
\[
\tang(\F,C,p)=\dim_{\C}\dfrac{\Og_p}{\langle f,Y(f)\rangle}.
\]
Moreover,
\[ 
\tang(\F,C)=\sum_{p\in\sing(\F)\cap C}\tang(\F,C,p).
\]
In addition, we have the following formulas~\cite[Lemme~2]{BR3}:
\begin{gather*}
N_{\F}\cdot C=\mathcal{X}(C)+\tang(\F,C),\\
T_{\F}\cdot C=C\cdot C-\tang(\F,C).
\end{gather*}

When $C$ is invariant by $\F$, for any $p\in C$, there is also the \emph{GSV index} $Z(\F,C,p)$, see \cite[\S 3]{BR2} and \cite[p.~24]{BR1} for details. Let $\{f=0\}$ be a local equation of $C$ around $p$, and let $\omega$ be an holomorphic 1-form generating $\F$ around $p$. Because $C$ is $\F$-invariant, we can factorize $\omega$ around $p$ as
\[
g\omega=hdf+f\eta,
\]
where $\eta$ is an holomorphic 1-form, $g$ and $h$ are holomorphic functions, and $h,f$ (and therefore $g,f$) are relatively prime, that is, $h$ (and therefore $g$) does not vanish identically on each branch of $C$. Thus, the \emph{GSV index} is defined as
\begin{align*}
Z(\F,C,p)&=\text{vanishing order of }\dfrac{h}{g}\bigg|_C\text{ at }p\\
&=\sum_{i}\text{vanishing order of }\dfrac{h}{g}\bigg|_{C_i}\text{ at }p
\end{align*}
Denote, in addition,
\[
Z(\F,C)=\sum_{p\in\sing(\F)\cap C}Z(\F,C,p)
\]
We also have the following formulas~\cite[Lemme~3]{BR3}:
\begin{align*}
N_{\F}\cdot C&=C\cdot C+Z(\F,C),\\
T_{\F}\cdot C&=\chi(C)-Z(\F,C).
\end{align*}

A \emph{fibration} on $X$ is a surjective holomorphic map $\varphi:X\to S$ over a Riemann surface such that $\varphi^{-1}(c)$ is connected for generic $c$. Let $CV(\varphi)$ be the set of critical values of $\varphi$. Given $c\in S$ denote $\varphi_c=\varphi^{-1}(c)$ and $V=\varphi^{-1}(CV(\varphi))$. By Ehresman fibration theorem~\cite{EhR}, the map 
\[
\varphi|_{X\setminus V}:X\setminus V\to S\setminus CV(\varphi) 
\]
is a locally trivial $C^{\infty}$ bundle. Then $\chi(\varphi_c)=\chi(\varphi_{c'})$, for every $c,c'\in S\setminus CV(\varphi)$. 
Therefore, the \emph{genus} of $\varphi$ is $\gen(\varphi)=\gen(\varphi_c)$, for any $c\in S\setminus CV(\varphi)$. When $\gen(\varphi)=0$ (respectively, $\gen(\varphi)=1$) the fibration $\varphi$ is called \emph{rational} (respectively, \emph{elliptic}). 

A foliation $\F$ on $X$ is \emph{tangent} to a fibration $\varphi$ if the leaves of $\F$ are the fibers of $\varphi$. Moreover, a foliation is called a \emph{rational} (respectively, \emph{elliptic}) \emph{fibration} if it is tangent to a rational (respectively, elliptic) fibration.

A foliation $\F$ on $X$ is called a \emph{Riccati foliation} (respectively, \emph{turbulent foliation}), if there exists a rational (respectively, elliptic) fibration whose generic fibers are transversal to $\F$.

Let $\F$ be a foliation on $X$. An \emph{holomorphic} (respectively, \emph{meromorphic}) \emph{first integral} of $\F$ is a non-constant holomorphic (respectively, meromorphic) map $f:X\to S$, where $S$  is a Riemann surface, and for all $c\in S$, $f^{-1}(c)$ is a union of leaves and singularities of $\F$.
We will assume that the generic level curve of $f$ is irreducible. This implies that any holomorphic first integral of a foliation is a fibration.

Let $f$ be a first integral of a foliation $\F$. When $f$ is holomorphic, the \emph{genus} of $f$ is the genus of $f$ considered as a fibration. On the other hand, when $f$ is meromorphic, the genus of $f$ is the genus of the fibration obtained after a finite number of blow-ups on the singularities of $f$.

Let $T$ be a two dimensional complex torus.  The following result is due to Ghys~\cite{E}.
\begin{proposition}\label{pro:new:ghys}
Let $\F$ be a regular foliation on $T$. Then there exists a covering $\pi:\C^2\to T$ such that $\F$ is induced by the closed 1-form $b(x)dx+dy$, where $b(x)$ is either constant or an elliptic function.
\end{proposition}

The set of linear foliations on $T$ is naturally identified with the one-dimensional projective space $\P H^0(T,\Omega^1_T)$. Let $I(T)$ be the set of linear foliations in $\P H^0(T,\Omega^1_T)$ which admit an holomorphic first integral and let $i(t)$ be its cardinality. 
The following proposition is a result due to Pereira and Pirio~\cite[Proposition~11.1]{JVPir}, adapted for our purposes.

\begin{proposition}\label{pro:new:vitorio}
If $i(T)\geq 3$ then $i(T)=\infty$ and there exists an elliptic curve $E=\C/\langle 1,\tau\rangle$ such that $T=E\times E$. Moreover, if $\omega_1,\omega_2$ are linearly independent 1-forms on $T$ admitting rational first integrals, then
\begin{equation}\label{eq:idx}
\{\lambda\in\C\::\:\omega_1+\lambda\omega_2\text{ has an holomorphic first integral}\}
=
\End(E)\otimes\Q,
\end{equation}
where $\End(E)\otimes\Q$ is either $\Q$ or $\Q(\tau)$.
\end{proposition}

\begin{proposition}\label{pro:24}
Let $X$ be a Hopf surface. Then
\[
\{\alpha\in\C\::\:dy-\alpha dx\text{ has an holomorphic first integral}\}
=\emptyset.
\]
\end{proposition}
\begin{proof}
First assume that $X$ is a primary Hopf surface, that is, it has the form
\[
X=\dfrac{\C^2\setminus\{(0,0)\}}{\Gamma}
\]
where $\Gamma$ is generated by $f(x,y)=(a x+\lambda y^r,b y)$, with $r\in\Z$, $r> 0$, $a,b\in\C$, $0<|a|\leq |b|<1$ and either $\lambda=0$ or $a=b^r$. Straightforward calculations show that the iterations of $f$ take the form
\[
f^{n}(x,y)=(a^n x+n\lambda a^{n-1}y^r,b^n y),\qquad n\in\Z.
\]
We claim that, for $n\in\Z$, $(x,1)=f^n(z,1)$ if and only if $n=0$ and $x=z$. The if part is immediate. For the only if part, is enough to observe that $(x,1)=f^n(z,1)=(a^n z+n\lambda a^{n-1},b^n)$ and $1=b^n$ implies $n=0$ and $x=z$, since $|b|<1$.

Given $\alpha\in\C\setminus\{0\}$, let us consider $L_{\alpha}=\{\pi(x,\alpha x+1)\::\:x\in\C\}$ and a cross section $\Sigma=\{\pi(x',1)\::\:x'\in\C\}$. 
Note that 
\[
\pi(x,\alpha x+1)=\pi(x',1)\in \Sigma\cap L_{\alpha} \quad\Longleftrightarrow\quad x=\dfrac{b^n-1}{\alpha},\, x'=\dfrac{b^n-1}{a^n\alpha}-\dfrac{n\lambda}{a},
\]
for some $n\in\Z$. 
By our previous claim, we conclude that there are infinite, non-equivalent, elements of the form $\pi(x',1)\in L_{\alpha}\cap\Sigma$. Therefore $L_{\alpha}\cap\Sigma$ is infinite.
Now for $\alpha=0$, we consider $L_0=\{\pi(x,1)\::\:x\in\C\}$ and let $\Sigma=\{\pi(1,y)\::\:y\in\C\}$. 
In this case,
\[
\pi(x,1)=\pi(1,y)\in \Sigma\cap L_0\quad\Longleftrightarrow\quad x=\dfrac{1}{a^n}-n\lambda \dfrac{1}{a},\, y=b^n,
\]
for some $n\in\Z$. Again, by our previous claim, there are infinite, non-equivalent, elements of the form $\pi(x,1)\in L_{0}\cap\Sigma$. Therefore $L_{0}\cap\Sigma$ is infinite. The case when $\alpha=\infty$ is analogous.
We conclude that $\F_{\alpha}$ must have infinite non-compact leaves, and so it does not have a first integral. 

When $X$ is a secondary Hopf surface, from Chapter 5, Section 18 in~\cite{BPV}, there exists a finite unramified cover that is a primary Hopf surface, that is, there exists an holomorphic $F:X_1\to X$, where $X_1$ is a primary Hopf surface. The proposition follows.
\end{proof}

\section{Pencils of foliations}\label{sec:pencils}
Most of the following definitions and properties were introduced by Lins Neto in~\cite{LN2}.
\subsection{First definitions}\label{sec:definitions}
Let $\F$ and $\G$ be two distinct foliatons on a compact complex surface $X$
with isolated singularities 
and whose normal fibers bundles are isomorphic, which will be denoted as
$N_{\F}=N_{\G}$. Then there exist an open covering $\U=\{U_i\}_{i\in I}$ of $X$ and collections $(\omega_i)_{i\in I}$, $(\eta_i)_{i\in I}$ and $(g_{ij})_{ij}$ such that
\begin{enumerate}
 \item $\omega_i$ and $\eta_i$ are holomorphic 1-forms on $U_i$, which define $\F$ and $\G$ on $U_i$, respectively,
 \item if $U_{ij}=U_i\cap U_j\neq\emptyset$ then $\omega_i=g_{ij}\omega_j$ and $\eta_i=g_{ij}\eta_j$, on $U_{ij}$.
\end{enumerate}
Using item~(2), given $\alpha\in\Ce$, the collections $\{U_{i},\omega_i+\alpha\eta_i,g_{ij}\}_{i,j\in I}$ define a foliation $\F_{\alpha}$. Thus, we obtain a linear family of foliations $\Pc(\F,\G)=\{\F_{\alpha}\}_{\alpha\in\Ce}$, which is called the \emph{pencil} generated by $\F$ and $\G$. 
Note that $\F_0=\F$ and $\F_{\infty}=\G$. 

The \emph{tangency set} between $\F$ and $\G$, denoted by $\tang(\F,\G)$, is defined as
\[
\tang(\F,\G)\cap U_i=\{\omega_i\wedge \eta_i=0\}.
\]

Given a pencil $\Pc=\Pc(\F,\G)$, it is straightforward to observe that, for every $\alpha,\beta\in\Ce$, $\alpha\neq\beta$,
\[
\tang(\F,\G)=\tang(\F_{\alpha},\F_{\beta}).
\]
This allows us to consider the \emph{tangency set} $\Delta(\Pc)$ of $\Pc$ as $\Delta(\Pc)=\tang(\F,\G)$. 
It is easy to prove that,
\begin{equation}\label{eq:deltapc}
\Delta(\Pc)=\bigcup_{\alpha\in\Ce}\sing(\F_{\alpha}).
\end{equation}
In particular, $\Delta(\Pc)=\emptyset$ if, and only if, $\sing(\F_{\alpha})=\emptyset$, for all $\alpha\in\Ce$.
Let $[\Delta(\Pc)]$ be the divisor defined by $\Delta(\Pc)$ then~\cite[Lemme~4]{BR3}
\begin{equation}\label{eq1.1}
\Og([\Delta(\Pc)])=T_{\F}^*\otimes N_{\G}. 
\end{equation}

Let $\Pc=\Pc(\F,\G)$ be a pencil. Consider the sets
\[
NI(\Pc)=\{\alpha\in\ovl{\C}\::\:\F_{\alpha}\text{ has non-isolated singularities}\}
\]
and $IS(\Pc)=\ovl{\C}\setminus NI(\Pc)$. From equation~\eqref{eq:deltapc}, we conclude that $NI(\Pc)$ is finite, since $\F=\F_0$ has isolated singularities and $\Delta(\Pc)$ is an analytic set. 
Moreover, for every $\alpha,\beta\in IS(\Pc)$, $\alpha\neq\beta$, we have $\Pc(\F,\G)=\Pc(\F_{\alpha},\F_{\beta})$ and $N_{\F_{\alpha}}=N_{\F_{\beta}}$.

The fact that $NI(\Pc)$ is finite implies that the foliations in $\Pc$ can have a singular set of codimension one, as shown by the following example.

\begin{example}\label{ex:1}
Let $X=\P^1\times\P^1$ and let $(x,y,\C^2)$ be a coordinate system on $X$. Let $\F$ and $\G$ two foliations on $X$ such that $\F|_{\C^2}$ is induced by the 1-form $dy=0$ and  $\G|_{\C^2}$ is induced by the 1-form $dx=0$.  It is not difficult to note that, if $\Pc=\Pc(\F,\G)=\{\F_{\alpha}\}_{\alpha\in\Ce}$ then $\F_{\alpha}$ is defined in $\C^2$ by 
$dy+\alpha dx=0$.
Thus, taking coordinates $(x,t,\C^2)$, $(u,y,\C^2)$ and $(u,t,\C^2)$, with $u=1/x$ and $t=1/y$, $\F_{\alpha}$ is respectively defined by
\[
dt-\alpha t^2dx=0,\quad -u^2dy+\alpha du=0,\quad u^2dt+\alpha t^2 du=0.
\]
Note that $N_{\F}\simeq N_{\G}$, since
\[
 dy+\alpha dx=\dfrac{1}{t^2}(dt-\alpha t^2dx)=\dfrac{1}{u^2}(-u^2dy+\alpha du)=2\dfrac{1}{t^2u^2}(u^2dt+\alpha t^2du).
\]

In addition,
\[
\sing(\F_{\alpha})=
\begin{cases}
\{\infty\} \times\P^1,&\text{if }\alpha=0,\\
\P^1\times \{\infty\},&\text{if }\alpha=\infty,\\
\{(\infty,\infty)\},&\text{if }\alpha\neq 0,\infty,
\end{cases}
\]
and 
\[
\Delta(\Pc)=\sing(\F_0)\cup\sing(\F_{\infty}).
\]
\end{example}

The following are examples of pencils with empty tangency set.

\begin{example}\label{ex:2}
Let $X=\C^2/\Gamma$ be a complex torus and let $\pi:\C^2\to X$ be the natural projection. Let $\Pc$ be the pencil on $X$ induced by the family of 1-forms $dx+\alpha dy$. Clearly $\sing(\F_{\alpha})=\emptyset$, for all $\alpha\in\Ce$, and therefore $\Delta(\Pc)=\emptyset$.
\end{example}
\begin{example}\label{ex:3}
Let $X$ be the Hopf surface defined by $\dfrac{\C^2\setminus\{(0,0)\}}{\langle f\rangle}$, with $f(x,y)=(x/2,y/2)$, and let $\pi:\C^2\setminus\{(0,0)\}\to X$ be the natural projection. In $\C^2\setminus\{(0,0)\}$ consider $\Pc$ defined by the 1-forms $dx+\alpha dy=0$ and let $\Pc_{*}=\pi_*(\Pc)$ be the induced pencil on $X$. Then $\Delta(\Pc_*)=\emptyset$.
\end{example}

\subsection{Curvature of a pencil}\label{sec:curvature}
Let $\Pc=\{\F_{\alpha}\}_{\alpha\in\Ce}$ be a pencil on a compact complex surface $X$, defined by $(\omega_i+\alpha\eta_i,U_i,g_{ij})_{i,j\in I}$. Given $i\in I$ we can find a unique 1-form $\theta_i$, meromorphic on $U_i$ and holomorphic on $U_i\setminus\Delta(\Pc)$, such that
\[
d\omega_i=\theta_i\wedge \omega_i,\qquad d\eta_i=\theta_i\wedge \eta_i.
\]
Hence we can define a 2-form $\Theta$ on $X$, holomorphic on $X\setminus\Delta(\Pc)$, such that
\[
\Theta|_{U_i\setminus\Delta(\Pc)}=d\theta_i.
\]
From now on, $\Theta=\Theta(\Pc)$ will be called the \emph{curvature} of $\Pc$. In addition, a pencil $\Pc$ is \emph{flat} whenever its curvature is zero, that is $\Theta(\Pc)\equiv 0$. 

\begin{remark}
The pencils $\Pc$ and $\Pc_*$ respectively defined in Examples~\ref{ex:2} and \ref{ex:3} are flat. This is a consequence of Proposition~\ref{pro:lnflat}, since both have an empty tangency set.
\end{remark}

The following lemma shows us that $X$ admits an \emph{holomorphic projective structure}~\cite{MR575449} outside $\Delta(\Pc)$.
\begin{lemma}[{\cite[Lemma 2.1.4]{LN2}}]\label{17}
Let $\Pc=\{\F_{\alpha}\}_{\alpha\in\Ce}$ be a flat pencil on $X$. Given $p\in X\setminus\Delta(\Pc)$, there exists a system of coordinates $(x,y,U)$,
where $U$ is a neighborhood of $p$ and $(x,y):U\to\C^2$, such that, for any $\alpha\in\Ce$, $\F_{\alpha}$ is defined on $U$ by the linear 1-form
\[
dy+\alpha dx.
\]
Furthermore, if $(u,v,V)$ is another system of coordinates, where $U\cap V$ is non-empty and convex and $\F_{\alpha}$ is defined by $dv+\alpha du$ then $du=\lambda dx$ and $dv=\lambda dy$, for some $\lambda\in\C$, $\lambda\neq 0$.
\end{lemma}

\subsection{On the components of the pencil's tangency set}

Let $\Pc=\{\F_{\alpha}\}_{\alpha\in\Ce}$ be a pencil on compact complex surface $X$ and let $C$ be an irreducible component of $\Delta(\Pc)$. Then $C$ is called \emph{invariant for the pencil} if $C$ is invariant for $\F_{\alpha}$, for all $\alpha\in\Ce$. Moreover, $\Delta(\Pc)$ is called \emph{invariant}, if every irreducible component of $\Delta(\Pc)$ is invariant for the pencil.
The set of parameters $\alpha\in\Ce$ for which $\F_{\alpha}$ admit an holomorphic first integral will be denoted by $I_p(\Pc)$.

\begin{example}\label{ejmmodelo}
Each one of the following examples defines a flat pencil whose tangency set is invariant. 
These examples can be found in~\cite{LN1} and \cite[Example 3]{LN2}. See also~\cite[Example 1.6]{LN3}.

\begin{enumerate}
\item Degree two pencil, defined as
$$
\Pc_2\:\begin{cases}
w_1=(4x-9x^2+y^2)dy-(6y-12xy)dx,\\
\eta_1=(2y-4xy)dy-3(x^2-y^2)dx,
\end{cases}
$$
with $\Delta(\Pc_2)=-4y^2+4x^3+12x y^2-9 x^4-6x^2y^3-y^4$.
\item Degree three pencil, defined as
$$
\Pc_3\:\begin{cases}
w_4=(-4x+x^3+3 x y^2)dy-2 y(y^2-1)dx,\\
\eta_4=(x^2y-y^3)dy-2 x(y^2-1)dx,
\end{cases}
$$
with $\Delta(\Pc_3)=(y^2-x)\Big(y-\frac{x^2}{4}\Big)(y^3-x^3+3xy+1)$.
\item Degree three pencil, defined as
$$
\Pc_3'\:\begin{cases}
w_2=(-x+2 y^2-4x^2 y+x^4)dy-y(-2-3xy+x^3)dx,\\
\eta_2=(2y-x^2+xy^2)dy-(3xy-x^3+2y^3)dx,
\end{cases}
$$
with $\Delta(\Pc_3')=(x^3-1)(y^3-1)(x^3-y^3)$.
\item Degree four pencil, defined as
$$
\Pc_4\:\begin{cases}
w_3=(x^3-1)xdy-(y^3-1)ydx,\\
\eta_3=(x^3-1)y^2dy-(y^3-1)x^2dx,
\end{cases}
$$
with $\Delta(\Pc_4)=(y^2-1)(x+2+y^2-2x)(x^2+y^2+2x)$.
\end{enumerate}
\end{example}

\begin{definition}
Let $\Pc=\{\F_{\alpha}\}_{\alpha\in\Ce}$ be a pencil on compact complex surface $X$ such that $\F_{\infty}$ has an holomorphic first integral $f:X\to S$. The pencil $\Pc$ is called \emph{transversal} to $f$ if there exists $\beta \in IS(\Pc)$ such that the generic fibers of $f$ are $\F_{\beta}$-transversal.
\end{definition}
In particular, note that if $\Pc$ is transversal to $f$ then, for all $\alpha\in IS(\Pc)$, the generic fibers of $f$ are also $\F_{\alpha}$-transversal.

\begin{lemma}\label{lem:finally}
Let $\Pc$ be a pencil on a compact complex surface $X$ such that $\F_{\infty}$ has an holomorphic first integral $f:X\to S$. If $\Delta(\Pc)$ is invariant then $\Pc$ is transversal to $f$.
\end{lemma}
\begin{proof}
Let $\alpha\in IS(\Pc)$.
For any generic regular fiber $F$ of $f$ we have $F\cap 
\sing(\F_{\alpha})=\emptyset$ and $F$ is transversal to $\Delta(\Pc)$. Therefore, $\tang(\F_{\alpha},F)=0$ for any generic regular fiber.
\end{proof}

\begin{remark}
Under the conditions of Lemma~\ref{lem:finally}: if $\Delta(\Pc)$ is invariant then
 \[
\Delta(\Pc)= \bigcup_{j=1}^{k} f^{-1}(c_j)\cup  \bigcup_{l=1}^{r} C_l,
\]
where $c_1,\ldots,c_k\in S$ and $C_l\subset\sing(\F_{\infty})$, for $l=1,\ldots,r$.
\end{remark}

\subsection{Holonomy of a pencil}\label{sec:holonomy}

Let $\Pc=\{\F_{\alpha}\}_{\alpha\in\Ce}$ be a pencil on a compact complex surface $X$ such that $\F_{\infty}$ has an holomorphic first integral $f:X\to S$ and $\Delta(\Pc)$ is invariant.

We will follow the steps of~\cite[\S~2]{LN2}. Throughout this section, we will denote by $T_{c}$ the level curve $f^{-1}(c)\subset X$.
Since $\Delta(\Pc)$ is invariant, by Lemma~\ref{lem:finally}, $\Pc$ is transversal to $f$. 
Then, there exist open sets $W=S\setminus\{c_1,\ldots,c_k\}$ and $U=X\setminus f^{-1}(\{c_1,\ldots,c_k\})$ such that, for any $\alpha\in IS(\Pc)\setminus\{\infty\}$, the foliation $\F_{\alpha}$ is transversal to the fibers of $f$ in all points of the set $U$.
By Ehresmann's theory of foliations transversal to the fibers of a fibration~\cite{EhR}, for any $\alpha\in IS(\Pc)\setminus\{\infty\}$, we have the associated \emph{holonomy representation} $\Hol_{\alpha}$ of $\F_{\alpha}$,
\[
\Hol_{\alpha}:\Pi_1(W,c)\to \diff(T_c),
\]
 where $c\in W$, with the following properties:
\begin{enumerate}
\item For any leaf $L$ of $\F_{\alpha}|_U$, $f|_L:L\to W$ is a covering map.
\item Given a (regular) fiber $T_c$ and a closed curve $\gamma:[0,1]\to W$ with $\gamma(0)=\gamma(1)=c$, that is, $[\gamma]\in \Pi_1(W,c)$, the diffeomorphism $\Hol_{\alpha}([\gamma]):T_c\to T_c$ is defined as follows: take $p\in T_c$ and let $L_{\alpha}(p)$ be the leaf of $\F_{\alpha}$ passing through $p$.
Since $f|_{L_{\alpha}(p)}:L_{\alpha}(p) \to W$ is a covering map, there exists a unique curve $\tilde{\gamma}$ on $L_{\alpha}(p)$ such that $f \circ \tilde{\gamma}=\gamma$ and
$\tilde{\gamma}(0)=p$. Now consider $\Hol_{\alpha}([\gamma])(p)=\tilde{\gamma}(1)$.
\item Denote by $G_{\alpha}$ the subgroup $\Hol_{\alpha}(\Pi_1(W,c))$ of $\diff(T_c)$, then
\[
 L\cap T_c=\{h(q): h \in G_{\alpha}\}.
\]
The group $G_{\alpha}$ is called the \emph{global holonomy group} of $\F_{\alpha}$.
\end{enumerate}
\begin{remark}
Since the foliation $\F_{\alpha}$ and the fibration $f|_U$ are holomorphic, it follows that all elements $h\in G_{\alpha}$ are automorphisms of the curve $T_c$. 
\end{remark}

Fixed $\beta \in IS(\Pc)$ and $p \in T_c$, by item~(2) above and Lemma~\ref{17}, there exists an open covering $\{U_n\}_{n=1}^m$ of $\widetilde{\gamma}_{\beta}[0,1]$ and a system of coordinates $(X_n,Y_n)$ on $U_n$ such that $U_n\cap U_{n+1}$ is connected and $\F_{\beta}$ is represented in $U_n$ by the 1-form
\begin{equation}\label{eq:10}
\omega_{n,\beta}=dY_n+\beta dX_n.
\end{equation}
Note that $\F_0|_{U_n}\::\:dY_n=0$, and $\F_{\infty}|_{U_n}\::\:dX_n=0$.

From~\cite{LN2}, there exists $\eps>0$ and a cross section $\Sigma_0$ to $\F_{\beta}$ at $q_0\in T_c$
such that $h_{\alpha}\in G_{\alpha}$ takes the form
\begin{equation}\label{eq:new:12}
h_{\alpha}|_{\Sigma_0}(q_0)=Y_m^{-1}(\lambda Y_1(q_0)+a\alpha +b),
\end{equation}
for all $\alpha$ in a certain open disk $D_{\eps}(\beta)$.

\begin{remark}
Also in~\cite{LN2}, Lins Neto proved that when $\gen(f)=1$, the generators of the global holonomy group $G_{\alpha}$, $f_{j,\alpha}$, take the form $f_{j,\alpha}(z)=\lambda_jz+a_j\alpha+b_j$, $j=1,\ldots,k$, where $z$ is the uniformizing coordinate of the torus.  We will deal with the case $\gen(f)=0$ in Section~\ref{sec:invariant}.
\end{remark}

\section{Classification of pencils with empty tangency set}\label{sec:caseempty}
In this section, we will deal with pencils $\Pc$ over a compact complex surface $X$ such that $\Delta(\Pc)=\emptyset$.  We begin with the following result extracted from~\cite[Theorem~2]{LN1}.

\begin{proposition}\label{pro:lnflat}
Let $\Pc$ be a pencil with empty tangency set, defined on a compact complex surface. Then $\Pc$ is flat.
\end{proposition}

\begin{lemma}\label{lem3}
If there exists a pencil on $X$ with empty tangency set then $X$ is a minimal surface, that is, $X$ does not contain smooth rational curves with self-intersection $-1$.
\end{lemma}
\begin{proof}
Let $\Pc=\Pc(\F,\G)$ be a pencil on $X$ such that $\Delta(\Pc)=\emptyset$ and assume that $X$ is not minimal. Then $X$ contains a rational and smooth curve $C$ such that $C\cdot C=-1$.
From equation~\eqref{eq1.1}, $\Og_X{[\Delta(\Pc)]}=T_{\F}^*\otimes N_{\G}=\Og_X$, hence $T_{\F}=N_{\G}$. Therefore
\begin{equation}\label{eqq4}
 T_{\F}=N_{\F}=N_{\G}.
\end{equation}
Now, if $C$ were invariant by $\F$ then, by equation~\eqref{eqq4} and the intersection formulas~\cite[p. 25]{BR1} 
we would have
\[
T_{\F}\cdot C=\mathcal{X}(C)-Z(\F,C)=2 = N_{\F}\cdot C=C\cdot C+Z(\F,C)=-1,
\]
which is a contradiction. Analogously, if $C$ were invariant by $\G$ then $2=N_{\G}\cdot C=T_{\G}\cdot C=-1$, again a contradiction. Thus $C$ is neither invariant by $\F$ nor $\G$, therefore
\begin{equation}\label{eqq44}
T_{\F}\cdot C=C\cdot C-\tang(\F,C)=N_{\G}\cdot C=
\mathcal{X}(C)-\tang(\G,C).
\end{equation}
As $T_{\F}\cdot C=C\cdot C-\tang(\F,C)=T_{\G}\cdot C=C\cdot C-\tang(\F,C)$, we have $\tang(\F,C)=\tang(\G,C)$. Then in equation~\eqref{eqq44} we obtain $C\cdot C=-1=\mathcal{X}(C)=2$, again a contradiction. Thus, $X$ is minimal.
\end{proof}

\begin{lemma}\label{lem:new44}
Let $\Pc$ be a pencil with empty tangency set, on a compact complex surface $X$.
 Let $\F\in\Pc$ and let $\mathcal{H}$ be a foliation tangent to a fibration whose generic fibers are $\F$-transversal. Then $\mathcal{H}\in\Pc$. 
\end{lemma}
\begin{proof}
Let $F$ be a regular fiber of $\mathcal{H}$, not invariant by $\F$, such that $\tang(\F,F)=0$. Given $p\in F$, by Lemma~\ref{17}, there exists a system of coordinates $(x,y,U)$, with $p\in U$ and $x(p)=y(p)=0$, where $\Pc$ is defined on $U$ by $dx+\alpha dy$. Hence $T_pX=\langle\frac{\partial}{\partial x},\frac{\partial}{\partial y}\rangle$, then there exists $\beta\in\Ce$ such that $\tang(\F_{\beta},\mathcal{H})>0$. 
We claim that $F$ is invariant by $\F_{\beta}$. Otherwise,
\[
 \tang(\F_{\beta},F)=F\cdot F-T_{\F_{\beta}}\cdot F=F\cdot F-T_{\F}\cdot F=\tang(\F,F)=0,
\]
a contradiction. Therefore, $\mathcal{H}=\F_{\beta}$. 
\end{proof}
\begin{lemma}\label{lem:new45}
Let $\Pc$ be a pencil with empty tangency set, on a compact complex surface $X$. Then $\Pc$ does not contain foliations tangent to rational fibrations. 
\end{lemma}
\begin{proof}
 Let $\F_{\beta} \in \Pc$ be tangent to a rational fibration $\varphi$. Then, for any $\alpha\neq\beta$, $\F_{\alpha}$ is a Riccati foliation relative to $\varphi$. Take a fiber $F$ of $\varphi$, then there exists a system of coordinates $(x,y,U)$ such that $F\subset U\simeq D\times\P^1$ such that the projections $\pi_1:D\times \P^1\to D$ and $\pi_2:D\times\P^1\to\P^1$ define $\F_{\beta}|_U$ and $\F_{\alpha}|_U$, respectively. Hence, the pencil $\Pc$ on $U$ is defined by the 1-form $dx+\alpha dy$, which has a non-empty tangency set, as Example~\ref{ex:1} shows. The lemma follows.
\end{proof}

The following theorem fully characterizes the compact surfaces which possess a pencil with empty tangency set.
\begin{theorem}\label{teodeltavacio}
Let $X$ be a compact complex surface which admits a pencil with empty tangency set. Then 
 $X$ is either a torus or a Hopf surface.
\end{theorem}
\begin{proof}
Let $\Pc=\Pc(\F,\G)$ be a pencil on $X$ such that $\Delta(\Pc)=\emptyset$.
From equation~\eqref{eqq4}, $T_{\F}=N_{\G}=N_{\F}$. Since $K_X=T_{\F}^*\otimes N_{\F}^*$, $K_X=N_{\F}^*\otimes N_{\F}^*$, which in turn implies $c_1(X)= 2 c_1(N_{\F})$ (cf.~\cite[p.~571]{BR3}).
On the other hand, $\Delta(\Pc)=\emptyset$ implies $\sing(\F_{\alpha})=\emptyset$, for all $\alpha \in \Ce$, so
\[
0=BB(\F)=N_{\F}^2=\dfrac{1}{4}c_1(X)^2,
\]
where the second equality follows from equation~\eqref{eq:newbbf}. 
Moreover, from equation~\eqref{eq:newmf}, 
we obtain
\begin{equation}\label{eqxx}
 c_2(X)=m(\F)-BB(\F)+ N_{\F}\cdot K_X=N_{\F}\cdot K_X=-2N_{\F}^2=0.
\end{equation}
Thus, from Lemma~\ref{lem3}, $X$ is a minimal surface with $c_1^2(X)=c_2(X)=0$. By Kodaira's Compact Surfaces Classification Theorem~\cite[Theorem 1.1]{BPV} 
we obtain $\Kod(X)<2$. 

From the proof of Theorem~\ref{teo:brintro} in~\cite{BR3}, Inoue surfaces just admit at most two regular foliations, so they cannot contain a pencil of foliations with empty tangency set. Therefore $X$ is not an Inoue surface.

Now assume that $X$ is neither a torus nor a Hopf surface. Then from Theorem~\ref{teo:brintro} and Lemma~\ref{lem:new44}, there exists a foliation $\F_{\alpha}\in\Pc$ which is an elliptic fibration, that is, $\F_{\alpha}$ is tangent to an elliptic fibration $\varphi:X\to S$, where
 $S$ is either an elliptic curve or isomorphic to $\P^1$. Note that, given any $\F_{\beta}\in\Pc$, $\beta\neq \alpha$, $\F_{\beta}$ is transversal to $\varphi$. Thus $\varphi$ is a principal fiber bundle with locally constant transition functions. From~\cite[p 197]{BPV}, since $X$ is not a torus,  $S$ is not an elliptic curve and $S=\P^1$.

Given $\F_{\beta}\in\Pc$, the holonomy group of $\F_{\beta}$ is also trivial, because $\Pi_1(\P^1)$ is trivial. This implies that for any leaf $L$ of $\F_{\beta}$, $\varphi|_L:L\to\P^1$ is a biholomorphism.
Therefore $\F_{\beta}$ is a rational fibration, a contradiction with Lemma~\ref{lem:new45}.
\end{proof}

\begin{theorem}\label{teo:linearfol}
Let $\Pc$ be a pencil of foliations with empty tangency set, on a compact complex surface $X$. Then $X$ is either a torus or a Hopf surface and $\Pc$ is generated by linear foliations.
\end{theorem}
\begin{proof}
Let $\Pc=\Pc(\F,\G)$. By Theorem~\ref{teodeltavacio}, $X$ is either a torus or a Hopf surface.  First we assume that $X$ is a torus. By Proposition~\ref{pro:new:ghys}, there exists a covering $\pi:\C^2\to X$ such that $\F$ and $\G$ are induced by the closed 1-forms $\omega=b(x)dx+dy$ and $\eta=c(x)dx+dy$, respectively, where $b$ and $c$ are either constants or elliptic functions. Let $x_0$ be a regular point of $b$ and $c$, then there exists a neighborhood $U$ of $\pi(x_0,0)$ such that 
\[
\Delta(\Pc)\cap U=\{\pi(x,y)\in U\::\:b(x)=c(x)\}=\emptyset.
\]
Since $b(x)\neq c(x)$, $b$ and $c$ must be both constants and, in particular, $\Pc$ is a pencil generated by linear foliations.

Now assume that $X$ is a Hopf surface. Note that $X$ is not elliptic, that is, it does not admit an elliptic fibration, otherwise by~\cite[p. 585]{BR3}, $\F$ and $\G$ would be defined in a neighborhood of $\C^2\setminus\{(0,0)\}$ by a quasihomogeneous vector field, so $\Delta(\Pc)\neq\emptyset$. 
Since $X$ is not elliptic, by~\cite[p. 586]{BR3}, in $\C^2\setminus\{(0,0)\}$ each foliation is either linear or can be extended to a foliation with a saddlepoint singularity at the origin, which cannot be possible since the tangency set is non-empty.
\end{proof}
Theorems~\ref{teodeltavacio} and \ref{teo:linearfol} imply that the only surfaces which admit pencils with empty tangency set are tori and Hopf surfaces. 
We now characterize $I_p(\Pc)$ for pencils with empty tangency set defined on these surfaces.

The following corollary is consequence from Proposition~\ref{pro:new:vitorio}.
\begin{corollary}\label{coro:ipptoro}
Let $X$ is a torus and let $\Pc$ be a pencil on $X$ with empty tangency set. If $\# I_p(\Pc)\geq 3$ then $X=E\times E$, with $E=\C/\langle 1,\tau\rangle$ and $I_p(\Pc)\setminus\{\infty\}$ is either $\Q$ or $\Q(\tau)$, only up to a reparametrization of the parameter space of the pencil.
\end{corollary}

The following corollary is consequence of Proposition~\ref{pro:24}.
\begin{corollary}\label{coro:ipphopf}
Let $X$ be a Hopf surface and let $\Pc$ be a pencil on $X$ with empty tangency set. Then $I_p(\Pc)$ is empty.
\end{corollary}

\section{Classification of pencils with invariant and non-empty tangency set}\label{sec:invariant}

The following lemma can be found in~\cite[Lemma~3.2.1]{LN3}.
\begin{lemma}\label{lem:new:22}
Let $\F$ and $\G$ be foliations on a compact complex surface $X$ with isolated singularities and isomorphic tangent bundles.
Assume that $\F$ has an holomorphic first integral $f:X\to S$, where $S$ is a compact Riemann surface. 
Thus,
\begin{enumerate}
 \item If $\gen(f)=0$ then $\F=\G$.
 \item If $\gen(f)=1$ and $\F\neq \G$ then $\G$ is turbulent relative to $f$.
 \item If $\gen(f)\geq 2$ and $\F\neq\G$ then $\tang(\G,F)>0$, for any regular fiber $F$ of $f$, non-invariant by $\G$.
\end{enumerate}
\end{lemma}

\begin{proposition}
Let $\Pc=\{\F_{\alpha}\}_{\alpha\in\Ce}$ be a pencil on a compact complex surface $X$ such that $\F_{\infty}$ has an holomorphic first integral $f:X\to S$. If $\infty\in IS(\Pc)$ then $\gen(f)\geq 1$. In addition:
\begin{enumerate}
 \item If $\gen(f)=1$ then there exist $c_1,\ldots,c_k\in S$ such that $\ds\Delta(\Pc)\subset\bigcup_{j=1}^k f^{-1}(c_j)$.
 \item If $\gen(f)\geq 2$ then $\Delta(\Pc)$ must have a non-invariant component.
\end{enumerate}
In particular, $\Delta(\Pc)$ is invariant if, and only if, $\gen(f)=1$.
\end{proposition}
\begin{proof}
Without loss of generality we may assume that $0$ also belongs to $IS(\Pc)$. As $\F_0\neq \F_{\infty}$, by Lemma~\ref{lem:new:22}, we obtain that $\gen(f)\geq 1$.

Let $\alpha\in IS(\Pc)$, there exists a regular fiber $F$, non-invariant by $\F_{\alpha}$, such that $F\cap \Delta(\Pc)\neq \emptyset$ and $F\cap \sing(\F_{\alpha})=\emptyset$. This implies that
\begin{equation}\label{eq:tangt0}
\tang(\F_{\alpha},F)>0.
\end{equation}
Assume $\gen(f)=1$.  We claim that every component $C$ of $\Delta(\Pc)$ is contained in some fiber of $f$. Assume otherwise, so $C$ is a non-invariant component of $\Delta(\Pc)$.   On the other hand,
\[
F\cdot F-\tang(\F_{\alpha},F)=T_{\F_{\alpha}}\cdot F=T_{\F_{\infty}}\cdot F=\chi(F)-Z(\F_{\infty},F)=0,
\]
where the last equality follows since $\gen(F)=1$ and $F\cap \sing(\F_{\alpha})=\emptyset$.
Hence $\tang(\F_{\alpha},F)=0$, a contradiction with~\eqref{eq:tangt0}.  Therefore, any component of $\Delta(\Pc)$ is contained in some fiber of $f$, so there exist $c_1,\ldots,c_k\in S$ such that 
\[
\Delta(\Pc)\subset \bigcup_{j=1}^{k} f^{-1}(c_j).
\]

Now assume that $\gen(f)\geq 2$, and let $F$ be a generic regular fiber of $f$, non-invariant by $\F_{\alpha}$. 
If $\Delta(\Pc)$ is invariant, then there exist $c_1,\ldots,c_k\in S$ such that $\Delta(\Pc)\subset\ds\bigcup_{j=1}^kf^{-1}(c_j)$ and $F\cap\Delta(\Pc)=\emptyset$. In particular,
$\tang(\F_{\alpha},F)=0$, again a contradiction with~\eqref{eq:tangt0}. This implies item~(2).
\end{proof}

\begin{corollary}\label{teo:teomain}
Let $\Pc=\{\F_{\alpha}\}_{\alpha\in\Ce}$ be a pencil on a compact complex surface $X$ such that $\F_{\infty}$ has an holomorphic first integral $f:X\to S$. If $\infty\in IS(\Pc)$ then the following are equivalent:
\begin{enumerate}
 \item $\gen(f)=1$,
 \item $\Delta(\Pc)$ is invariant,
 \item there exist $c_1,\ldots,c_k\in S$ such that $\Delta(\Pc)\subset \ds\bigcup_{j=1}^{k} f^{-1}(c_j)$.
\end{enumerate}
In particular, in this case, $\Pc$ is transversal to $f$. 
\end{corollary}

\begin{lemma}\label{lemagen01}
Let $\Pc=\{\F_{\alpha}\}_{\alpha\in\Ce}$ be a pencil on a compact complex surface $X$ such that $\F_{\infty}$ has an holomorphic first integral $f:X\to S$ and $\Delta(\Pc)$ is invariant.
Let $\alpha\in IS(\Pc)$ such that $Z(\F_{\alpha},C)\geq 1$, for all components $C$ of $\Delta(\Pc)$, $\F_{\alpha}$ has a meromorphic local first integral in every singularity, and $G_{\alpha}$ is finite. Then $\alpha\in I_p(\Pc)$.
\end{lemma}
\begin{proof}
By \cite{E2}, to prove that $\alpha\in I_p(\Pc)$, it is enough to prove that $\F_{\alpha}$ has infinitely many compact invariant curves.  
Let $L$ be a leaf of $\F_{\alpha}$ not contained in $\Delta(\Pc)$. We assert that $\ovl{L}\setminus L\subset \sing(\F_{\alpha})$. Take $x\in \ovl{L}\setminus L$. 
Since $G_{\alpha}$ is finite, there exists a component $C$ of $\Delta(\Pc)$ such that $x\in C$. Now $Z(\F_{\alpha},C)\geq 1$ implies that $C\cap\sing(\F_{\alpha})\neq\emptyset$. Let $p\in C\cap\sing(\F_{\alpha})$, then, using our hypotheses, $\F_{\alpha}$ has a meromorphic first integral defined on a neighborhood $U$ of $p$. We may assume, without loss of generality, that $x\in U$, hence $x\in \sing(\F_{\alpha})$.
\end{proof}

\begin{proposition}\label{pro:new54}
Let $\Pc=\{\F_{\alpha}\}_{\alpha\in\Ce}$ be a pencil on a compact complex surface $X$ such that $\F_{\infty}$ has an holomorphic first integral $f:X\to S$. If $\Delta(\Pc)$ is invariant then $\gen(f)\leq 1$.
Moreover, let $\alpha\in IS(\Pc)$, $\alpha\neq \infty$.
\begin{enumerate}
 \item If $\gen(f)=0$ then $\infty\notin IS(\Pc)$ and $\F_{\alpha}$ is a Riccati foliation with respect to $f$.
 \item If $\gen(f)=1$ then $\F_{\alpha}$ is a turbulent foliation relative to $f$.
\end{enumerate}
\end{proposition}
\begin{proof}
On the contrary, assume that $\gen(f)\geq 2$. Since $\Delta(\Pc)$ is invariant, $\Pc$ is transversal to $f$. 
From Section~\ref{sec:holonomy}, given $\gamma\in\Pi_1(W,c)$, we obtain the holomorphic map 
\[
F_{\gamma}:IS(\Pc)\setminus\{\infty\}\times T_c\to T_c,\qquad F_{\gamma}(\alpha,p)=\Hol_{\alpha}(\gamma)(p),
\]
where $\Hol_{\alpha}$ is the holonomy representation.
Note that $F_{\gamma}(\alpha,\cdot)$ does not depend on $\alpha$, as $\aut(T_c)$ is finite. 
In particular, if $\alpha,\beta\in IS(\Pc)\setminus\{\infty\}$ then their global holonomy groups coincide. 
Using the same argument as \cite[p. 34]{LN3} we obtain a neighborhood $V$ of $\beta$ such that $\F_{\alpha}=\F_{\beta}$, for all $\alpha\in V$. This is a contradiction.

Finally, items (1) and (2) follow from the fact that $\Pc$ is transversal to $f$.
\end{proof}

From now on $\Pc=\{\F_{\infty}\}_{\alpha\in\Ce}$ is a flat pencil on a compact complex surface $X$ such that $\F_{\infty}$ has an holomorphic first integral $f:X\to S$ and $\Delta(\Pc)$ is invariant and non-empty.
From Proposition~\ref{pro:new54}, $\gen(f)\leq 1$. Theorem~\ref{teo:linsneto} deals with the case when $\gen(f)=1$. We now consider the case when $\gen(f)=0$.  Under this condition, the following proposition provides an explicit form of the elements of the global holonomy group.

\begin{theorem}\label{teo:new:51}
Let $\Pc=\{\F_{\alpha}\}_{\alpha\in\Ce}$ be a flat pencil on a compact complex surface $X$ such that $\F_{\infty}$ has an holomorphic first integral $f:X\to \P^1$ and $\Delta(\Pc)$ is invariant. If $\gen(f)=0$ then, for any $\alpha\in IS(\Pc)$, the global holonomy group $G_{\alpha}$ of $\F_{\alpha}$ is finitely generated by $f_{1,\alpha},\ldots,f_{k,\alpha}$, where these generators could be either
\[
f_{j,\alpha}(z)=\lambda_jz+a_j\alpha +b_j
,\qquad \forall\,j=1,\ldots,k,
\]
or
\[
f_{j,\alpha}(z)=\exp(2\pi i(\mu_j\alpha+\nu_j))z,\qquad \forall\,j=1,\ldots,k.
\]
\end{theorem}
\begin{proof}
We will follow the notations used in Section~\ref{sec:holonomy}.
Since $\Delta(\Pc)$ is invariant, $\Pc$ is transversal to $f$ and, as $\gen(f)=0$, $\F_{\alpha}$ is a Riccati foliation, for all $\alpha\in IS(\Pc)$. In addition, since $\infty\notin IS(\Pc)$, there exists a component of $\Delta(\Pc)$ contained in $\sing(\F_{\infty})$ and 
\[
\Delta(\Pc)=\sum_{j=1}^kn_j[f^{-1}(c_j)]+\sum_{s=1}^lC_s',
\]
for $j=1,\ldots,k$, and $C_s'\subset\sing(\F_{\infty})$, for $s=1,\ldots,l$.

Now assume, without loss of generality, that $0\in IS(\Pc)$. From now on, for each $c\in W$, with $W$ as in Section~\ref{sec:holonomy}, $T_c$ will denote the level curve $f^{-1}(c)$.
Fix $j\in\{1,\ldots,k\}$, then there exists a neighborhood $D_j\subset W$ of $c_j$ such that 
\[
f^{-1}(D_j)\simeq D_j\times T_{c_j}\ni (x,y),
\]
where $T_{c_j}\simeq \P^1$. Moreover, since $\F_{0}$ is transversal to $f$, we may assume that the projections $\pi_1:D_j\times T_{c_j}\to D_j$ and $\pi_2:D_j\times T_{c_j}\to T_{c_j}$ respectively define the foliations $\F_{\infty}$ and $\F_0$ on $D_j\times T_{c_j}$. In particular, $dx=0$ and $dy=0$ respectively represent $\F_{\infty}$ and $\F_{0}$ on $D_j\times T_{c_j}$.

We now look for an explicit expression for the generators of $G_{\alpha}$. For this, fix $c\in W$. We first consider the case when $c\in D_j$. Recall that $f_{j,\alpha}=\Hol_{\alpha}(\gamma)$, for some curve $\gamma\in\Pi_1(W,c)$. Now $c\in D_j$ implies that $f_{j,\alpha}$ coincides with the holonomy of $\F_{\alpha}$ around $T_{c_j}$.
Since $\F_{\alpha}$ is a Riccati foliation, 
from the construction done in Section~\ref{sec:holonomy},
there exists an open covering $\{U_n\}_{n=1}^m$ of $\widetilde{\gamma}([0,1])$, where $\widetilde{\gamma}$ is a lifting of $\gamma$, such that $D_1,D_m\subset D$, $U_n\simeq D_n\times\P^1$, $D_n$ is holomorphic to a disc in $\C$, and $\F_{\alpha}|_{U_n}$ is given by 
\begin{equation}\label{eq:new:pxy}
\omega_{\alpha}=dy+\alpha P(x,y)dx,
\end{equation}
where $P(x,y)=\ds\sum_{i=0}^r a_i(x)y^i$, $r\leq 2$ and $a_i(x)$ holomorphic, for $i=0,\ldots,r$.

From equation~\eqref{eq:new:pxy} we obtain $d\omega_{\alpha}=\theta\wedge\omega_{\alpha}$, with $\theta=\dfrac{P_y}{P}dy$, for all $\alpha\in\Ce$. Since $\Pc$ is flat, there exist $(\nu_0,\nu_1,\nu_2)\in\C^3\setminus\{0\}$ and a holomorphic $a(x)$ such that $a_i(x)=\nu_i a(x)$, $i=0,\ldots,r$. This in turn implies that $P(x,y)=a(x)p_2(y)$, where $p_2(y)$ is a polynomial with degree at most two. Moreover, $d\omega_{\alpha}=\theta\wedge\omega_{\alpha}$, with $\theta=\dfrac{p_2'(y)}{p_2(y)}dy$, for all $\alpha\in\Ce$. Thus, after a change of coordinates in the fiber, we have two possibilities, either $C_1=\{y=\infty\}\subset\sing(\F_{\infty})$ is a component of $\Delta(\Pc)$ with multiplicity 2, or $C_1=\{y=0\}$ and $C_2=\{y=\infty\}$ are invariant components of $\Delta(\Pc)$ with multiplicity 1, and $C_1,C_2\subset \sing(\F_{\infty})$. We now address these two cases separately:
\begin{enumerate}
 \item If $C_1=\{y=\infty\}$ is a component of $\Delta(\Pc)$ with multiplicity 2 then $p_2(y)=\lambda$, for some $\lambda\neq 0$. In addition, 
 \[
 \omega_0+\alpha \omega_{\infty}=d\left(y+\alpha\lambda\int_{x_0}^x a(t)dt\right)=d(Y_n+\alpha X_n),
 \]
 where 
 \begin{equation}\label{eq:new:17}
 Y_n(x,y)=y+\text{constant},
\end{equation}
for all $n=1,\ldots,m$. Now, from equation~\eqref{eq:new:12},
\[
f_{j,\alpha}|_{\Sigma_0}(q_0)=Y_m^{-1}(\lambda Y_1(q_0)+a\alpha+b)
\]
and from equation~\eqref{eq:new:17}, we obtain
\[
f_{j,\alpha}(z)=\lambda_jz+a_j\alpha +b_j.
\]
\item If $C_1=\{y=0\}$ and $C_2=\{y=\infty\}$ are invariant components of $\Delta(\Pc)$ with multiplicity 1, then $p_2(y)=y$. In addition,
\[
\dfrac{\omega_0+\alpha\omega_{\infty}}{y}=\dfrac{dy}{y}+\alpha a(x)dx=d(Y_n+\alpha X_n),
\]
where $Y_n|_{\Sigma_n}(y)=\exp(2\pi i(\mu_j\alpha+\nu_j))y$, for $n=1,\ldots,m$. This implies $f_{j,\alpha}(z)=\exp(2\pi i(\mu_j\alpha+\nu_j))z$.
\end{enumerate}

On the other hand, when $c\notin D_j$, $f_{j,\alpha}$ is conjugated to one of the expressions obtained on items (1) and (2). Therefore, since $\Delta(\Pc)$ is invariant and $\F_{\alpha}$ is a Riccati foliation, 
\[
f_{j,\alpha}(z)=\lambda_jz+a_j\alpha +b_j
,\qquad \forall\,j=1,\ldots,k,
\]
or
\[
f_{j,\alpha}(z)=\exp(2\pi i(\mu_j\alpha+\nu_j))z,\qquad \forall\,j=1,\ldots,k.
\]
\end{proof}

\begin{lemma}\label{lemagen02}
Let $\Pc=\{\F_{\alpha}\}_{\alpha\in\Ce}$ be a  pencil on a compact complex surface $X$ such that $\F_{\infty}$ has an holomorphic first integral $f:X\to \P^1$ and $\Delta(\Pc)$ is invariant. Assume also that, for all $\alpha\in IS(\Pc)$, the global holonomy group of $\F_{\alpha}$ is given by 
\[
G_{\alpha}=\langle f_{j,\alpha}\rangle,\text{ with } 
f_{j,\alpha}(z)=\lambda_jz+a_j\alpha +b_j,\quad  j=1,\ldots, k.
\]
If $\#(I_p(\Pc)\cap IS(\Pc))\geq 2$
then $G_{\alpha}$ is finite, for all $\alpha \in IS(\Pc)$.
\end{lemma}
\begin{proof}
Without loss of generality we may assume that $0\in I_p(\Pc)\cap IS(\Pc)$ and 
\[
 G_{\alpha}=\langle\lambda_1 z+a_1\alpha+b_1, \ldots, \lambda_n z+a_n\alpha+b_n,z+a_{n+1}\alpha+b_{n+1}, \ldots,z+a_k\alpha+b_{k}\rangle
\]
where $\lambda_j\neq 1$, for $j\leq n$.
Since $0 \in I_p(\Pc)$, $f_{j,0}$ has finite order, for all $j$, $\lambda_j$ is a root of unity, for all $j\leq n$, and $b_j=0$, for $j\geq n+1$. Besides, as $G_0$ is abelian, the generators $f_{j,0}$ have a common fixed point $\dfrac{b_j}{1-\lambda_j}$, which does not depend on $j$. 
Now take $\beta\in I_p(\Pc)\cap IS(\Pc)$, $\beta\neq 0$, which exists, by hypotheses. Note that, for all $j\geq n+1$, $\beta a_j=0$, which implies $a_j=0$.
Since $G_{\beta}$ is abelian, the common fixed point $\dfrac{\beta a_j+b_j}{1-\lambda_j}$ of $f_{j,\beta}$ does not depend on $j$. This in turn implies that $\dfrac{a_j}{1-\lambda_j}$ also does not depend on $j$.
In particular, for any $\alpha\in IS(\Pc)$, $G_{\alpha}$ is abelian, because  $\dfrac{a_j}{1-\lambda_j}$ and  $\dfrac{b_j}{1-\lambda_j}$ do not depend on $j$.
We conclude that $G_{\alpha}$ is finite, for all $\alpha\in IS(\Pc)$, because $G_{\alpha}$ is an abelian group, finitely generated,  whose all its elements have finite order.
\end{proof}

\begin{theorem}
Let $\Pc=\{\F_{\alpha}\}_{\alpha\in\Ce}$ be a flat pencil on a compact complex surface $X$ such that $\F_{\infty}$ has an holomorphic first integral $f:X\to \P^1$ and $\Delta(\Pc)$ is invariant. If $\gen(f)=0$ then $X$ is a rational surface where one of the following hold:
\begin{enumerate}
 \item $I_p(\Pc)$ is finite,
 \item $IS(\Pc)\subset I_p(\Pc)$, or
 \item there exist $\beta\in I_p(\Pc)$ and $\lambda\in\C^*$ such that 
 \[
 I_p(\Pc)\cap IS(\Pc)=(\lambda \Q+\beta)\cap IS(\Pc).
 \]
\end{enumerate}

\end{theorem}
\begin{proof}
We can suppose that every fiber of $f$ is regular, up to blowing-down $\pi:X\to X_1$~\cite[p.~142]{BPV}. In particular, $g=f\circ\pi^{-1}:X_1\to \P^1$ is an holomorphic fiber bundle with typical fiber $\P^1$ and $X_1$ is a $n$-th Hirzebruch surface~\cite[p. 141]{BPV}. Therefore $X$ is rational, because $X_1$ is birrationally equivalent to $\P^2$.

 By Theorem~\ref{teo:new:51}, for any $\alpha\in IS(\Pc)$, the holonomy group $G_{\alpha}=\langle f_{1,\alpha},\ldots,f_{k,\alpha}\rangle$ is determined by either
 \begin{multline*}
 f_{j,\alpha}(z)=\lambda_j z+\alpha a_j+b_j,\,\forall\,j\in\{1,\ldots,k\},\\
 \text{or}\\f_{j,\alpha}(z)=\exp(2\pi i(\mu_j\alpha+\nu_j))z,\,\forall\,j\in\{1,\ldots,k\}.
 \end{multline*}

Assume that $I_p(\Pc)$ is infinite. In particular $I_p(\Pc)\cap IS(\Pc)$ is also infinite and, without loss of generality, we also may assume that $0\in I_p(\Pc)\cap IS(\Pc)$.

We divide our analysis in two cases. First assume that $f_{j,\alpha}(z)=\lambda_j z+\alpha a_j+b_j$, for all $j=1,\ldots,k$. Fix $\alpha\in IS(\Pc)$. 
By Lemma~\ref{lemagen02}  $G_{\alpha}$ is finite. Moreover, following the proof of Proposition~2 in~\cite{BR1}, since $0\in I_p(\Pc)$,  $\F_{\alpha}$ is a Riccati foliation and $\Delta(\Pc)$ is invariant, $\F_{\alpha}$ has  local first integral on their singularities and $Z(\F_\alpha, C) \geq 1$, for every component $C$ of $\Delta(\Pc)$. 
Therefore, by Lemma~\ref{lemagen01}, $\alpha \in I_p(\Pc)$.

On the other hand, assume that $f_{j,\alpha}(z)=\exp(2\pi i(\mu_j\alpha+\nu_j))z$, $j=1,\ldots,k$. In this case:
\begin{enumerate}
 \item If $\mu_j=0$ for all $j=1,\ldots,k$, take  $\alpha\in IS(\Pc)$.  Since $0\in I_p(\Pc)$, $\nu_j\in\Q$, for all $j=1,\ldots,k$, and $G_{\alpha}$ is finite.
 Since $\F_{\alpha}$ has a meromorphic local first integral in each singularity and  $ Z(\alpha, C) \geq 1$, for all component $C$ of $\Delta(\Pc)$, from Lemma~\ref{lemagen01}, $\alpha\in I_p(\Pc)$. 
 
 \item If there exists $j\in\{1,\ldots,k\}$ such that $\mu_j\neq 0$ then we may assume, without loss of generality, that $\mu_j\neq 0$, for $j=1,\ldots,r$ and $\mu_j=0$, for $j=r+1,\ldots,k$. Fix $\beta\in I_p(\Pc)\cap IS(\Pc)$, so $G_{\beta}$ is finite and $\mu_j\beta+\nu_j\in\Q$, for all $j=1,\ldots,r$. Let $\alpha\in IS(\Pc)$, $\alpha\neq\beta$, then
\begin{multline*}
G_{\alpha}\text{ is finite}\iff\\
\mu_j(\alpha-\beta)+\mu_j\beta+\nu_j=\mu_j\alpha+\nu_j\in\Q,\,\forall\,j=1,\ldots,r,\\
\iff \mu_j(\alpha-\beta)\in\Q,\,\forall\,j=1,\ldots,r,
\end{multline*}
In particular, for $i,j\in\{1,\ldots,r\}$, we obtain $\dfrac{\mu_i}{\mu_j}\in\Q$, that is,
\[
(\mu_1,\ldots,\mu_r)=\mu_1(1,q_2,\ldots,q_r),
\]
where $q_i\in\Q$, for $i=2,\ldots,r$. The implies, even when $\alpha=\beta$, that
\begin{align*}
G_{\alpha}\text{ is finite}&\iff \mu_j(\alpha-\beta)\in\Q,\,\forall\,j=1,\ldots,r,\\
&\iff \mu_1(\alpha-\beta)\in\Q\\
                           &\iff \alpha\in \mu_1^{-1}\Q+\beta.
\end{align*}
In particular, if $\alpha\in I_p(\Pc)\cap IS(\Pc)$ then $\alpha\in \mu_1^{-1}\Q+\beta$. 
Now take $\alpha\in(\mu_1^{-1}\Q+\beta)\cap IS(\Pc)$, hence $G_{\alpha}$ is finite. Since $\F_{\alpha}$ has a meromorphic local first integral in each singularity and  $Z(\alpha, C) \geq 1$, for every component $C$ of $\Delta(\Pc)$, from Lemma~\ref{lemagen01}, $\alpha\in I_p(\Pc)$. 
\end{enumerate}
\end{proof}

When $\infty\in IS(\Pc)\cap I_p(\Pc)$ and $\Delta(\Pc)$ is invariant, from Theorem~\ref{teo:linsneto}, we obtain the following proposition.

\begin{proposition}
Let $\Pc=\{\F_{\alpha}\}_{\alpha\in\Ce}$ be a flat pencil on a compact complex surface $X$ such that $\F_{\infty}$ is reduced and has an holomorphic first integral $f:X\to S$, 
and $\Delta(\Pc)$ is invariant. If $\infty\in IS(\Pc)$ and is a limit point of $I_p(\Pc)$ then:
\begin{enumerate}
 \item If $K_X\neq 0$ then there exist $\lambda\in\C^*$ and $a\in\C$ such that
 \[
 I_p(\Pc)=(\lambda(\Q\oplus\tau\Q)+a)\cup\{\infty\},
 \]
where $\tau$ is either $i$ or $e^{2\pi i/3}$.
\item If $K_X=0$ and $\gen(S)=0$ then $X$ is an algebraic torus. Moreover, assume $X=E\times E$, where $E$ is an elliptic curve, then 
\[
I_p(\Pc)=\End(E)\otimes\Q.
\]
\end{enumerate}
\end{proposition}

\begin{proof} 
Since $\Delta(\Pc)$ is invariant and $\infty\in IS(\Pc)$, by Corollary~\ref{teo:teomain}, $\gen(f)=1$. On the other hand, from the fact that $\F_{\infty}$ is reduced and $I_p(\Pc)$ has a limit point at $\infty$, we obtain that there exists a neighborhood $U$ of $\infty$ such that, for any $\alpha\in U$, $\F_{\alpha}$ has reduced singularities and holomorphic first integral. We now take any $\beta\in U$ such that $\F_{\beta}\neq\F_{\infty}$. By Theorem~\ref{teo:linsneto}, the proposition follows.
\end{proof}

\section{Acknowledgements}

We would like to thank the anonymous referees, for the comprehensive review and for the helpful comments and suggestions, which helped improved this work.

\bibliographystyle{abbrv}
\bibliography{refer}
\end{document}